\documentclass{article}

\usepackage{amssymb,amsmath,latexsym,amsthm,array}%or any usual package
\usepackage[english]{babel}

\makeatletter
\newcommand{\subjclass}[2][1991]{%
  \let\@oldtitle\@title%
  \gdef\@title{\@oldtitle\footnotetext{#1 \emph{Mathematics subject classification.} #2}}%
}
\newcommand{\keywords}[1]{%
  \let\@@oldtitle\@title%
  \gdef\@title{\@@oldtitle\footnotetext{\emph{Key words and phrases.} #1.}}%
}
\makeatother
\newtheorem{thm}{Theorem}[section]
\newtheorem{theorem}[thm]{Theorem}
\newtheorem{lemma}[thm]{Lemma}

\theoremstyle{definition}
\newtheorem{remark}[thm]{Remark}
\newtheorem{definition}[thm]{Definition}

\numberwithin{equation}{section}

\def\cA{\mathcal A}         \def\cI{\mathcal I} \def\cJ{\mathcal J}  \def\cL{\mathcal L}  \def\cN{\mathcal N}   \def\cP{\mathcal P}    \def\cU{\mathcal U}   

\def\fm{\mathfrak{m}} \def\fM{\mathfrak{M}}  \def\fS{\mathfrak{S}}
  
\def\Z{{\mathbb Z}} \def\R{{\mathbb R}}  \def\N{{\mathbb N}}

\def\PS{Piatetski-Shapiro }
\def\sh{\cA_c} % Shapiro Set
\def\sp{\cP_c}
\def\M{{\fM_{\kappa}}}
\def\m{{\fm_{\kappa}}}

\def\eps{{\varepsilon}}

\def\lm{\Lambda} %von Mangoldt
\def\Ga{\Gamma}

\def\su#1{\sum_{\substack{#1}}}
\def\bs#1{\begin{equation*} \begin{split} #1 \end{split} \end{equation*}}
\def\bsc#1{\begin{equation} \begin{split} #1 \end{split} \end{equation}}
\def\eqs#1{\begin{equation*} #1 \end{equation*}}
\def\eqn#1{\begin{equation} #1 \end{equation}}

\def\mult#1{
	\begin{multline*}
	#1
	\end{multline*}
}
\def\multn#1{
	\begin{multline}
	#1
	\end{multline}
}

\def\({\left(} \def\){\right)} \def\[{\left[} \def\]{\right]} 
\def\fl#1{\left\lfloor#1\right\rfloor} 
\def\le{\leqslant} \def\ge{\geqslant}
\def\e{\mathbf{e}}
\def\mod{\,\text{mod}\,}

\def \Sc{ S_{c,k} (\alpha,X)} 
\def \Tc{ T_{c,k} (\alpha,X)}
\def \Rc#1{ R_{s,k}^{(c)} ({#1})}

\begin{document}
\title%[Waring-Goldbach Problem with PS Primes]
{Waring-Goldbach Problem with Piatetski-Shapiro Primes}
\author{Y\i ld\i r\i m Akbal \\ Ahmet M.~G\"ulo\u glu
\thanks{Both authors were supported by T\"UB\.ITAK Research Grant no. 114F404}
\\Department of Mathematics \\Bilkent University \\06800 Bilkent, Ankara, TURKEY \\{\tt yildirim.akbal@bilkent.edu.tr} \\{\tt guloglua@fen.bilkent.edu.tr} \\
\date{\today}}
\maketitle
\begin{abstract}
In this paper, we exhibit an asymptotic formula for the number of representations of a large integer as a sum of a fixed power of Piatetski-Shapiro primes, thereby establishing a variant of Waring-Goldbach problem with primes from a sparse sequence.
\end{abstract}

\bigskip
\section{Introduction}
We define, for a natural number $k$, and a prime $p$, $\theta=\theta(p,k)$ to be the largest natural number such that $p^\theta \mid k$, and define $\gamma(p,k)$ by
\eqs{
	\gamma = \gamma (p,k) = 
	\bigg\{ 
	\begin{array}{r@{,\quad}l}
		\theta+2  & \text{if $p=2$ and $2\mid k$,} \\
		\theta+1  & \text{otherwise.}
\end{array}  }
We then put $K(k)=\prod_{(p-1) \mid k} p^{\gamma}$. In this work, %as an analog of Waring-Goldbach problem, 
we establish an asymptotic formula for the number of representations of a positive  
%study the problem of representing \emph{every} sufficiently large 
integer $\cN$ in the form 
\eqn{\label{representation}
	\cN = p_1^k+ \cdots + p_s^k, \qquad \text{with } p_1, \ldots, p_s \in \sp}
for $k \ge 3$, provided that $\cN$ is congruent to $s$ modulo $K(k)$, and $c>1$ takes values in a small interval depending on $s$ and $k$. Here, the set of primes
\eqs{
	%\label{PS}
	\sp = \{\fl{m^c} : \fl{m^c} \text{ is prime for some } m \in \N \} %\qquad (c>1)
}
is named after \emph{I.I.~Piatetski-Shapiro}, since he was the first to prove an analog of the Prime Number Theorem (cf.~\cite{Shapiro}) for $c \in (1,12/11)$.

\begin{theorem} \label{thm1}
	%For $k=3,4$, set $t = t(k) = 2^{k-1}$, and for $k >4$, 
	Let $t>0$ be any integer such that the inequality 
	\eqn{\label{2t_momentofT1}
		\int_0^1 \bigg| \sum_{1 \le n \le X} e^{2\pi i \alpha n^k} \bigg|^{2t} d \alpha < C X^{2t-k} \log^{\eta} X}
	holds for all $X>2$ with some constants $C = C(k,t)$ and $\eta=\eta(t,k) \ge 0$. Then, for any integer $s> 2t$, the number of representations $\Rc \cN$ of a positive integer $\cN$ as in \eqref{representation} satisfies 
	\eqs{\Rc \cN = \frac{\Gamma \( 1+ 1/(ck) \)^s}{\Gamma \( s/(ck) \)}\, \fS (\cN) \,\frac{\cN^{s/(ck) -1}}{\log^s \cN} + o \Bigl( \frac{\cN^{s/(ck) -1}}{\log^s \cN}  \Bigr) %\qquad (\cN \rightarrow \infty)
	}
	where $\fS (\cN)$ defined in \eqref{defs} is the singular series in the classical Waring-Goldbach problem, provided that $c$ is a fixed number satisfying 
	\eqn{\label{rangeofc}
		1< c <  1 + (s-2t)
		\left\{  
		\def\arraystretch{2.2}
		\begin{array}{ll}
			3 \min\Big\{\dfrac{1}{77s + 158 t}, \dfrac{1}{75s + 164t}\Big\}, & k =3, \\
			\dfrac{1}{(\nu-1)s + 2t\nu}, & k > 3,
		\end{array}
		\right.
	}
	where 
	\eqn{\label{nu}
		\nu  = \Bigg\{ 
		\def\arraystretch{1.6}
		\begin{array}{l@{,\quad\text{if }}l}
			k(k+1)^2 & 4 \le k  \le 11, \\
			\dfrac{2\fl{3k/2} (\fl{3k/2}^2-1)}{\fl{3k/2}-k} & k \ge 12.
		\end{array} %\right.
	}
\end{theorem}
By \cite[Lemmas 8.10 and 8.12]{Hua}, when $\cN \equiv s \pmod{K(k)}$, the singular series satisfies $\fS(\cN) \asymp 1$ for the values of $s$ given in Theorem \ref{thm1}. Thus, our theorem implies that all sufficiently large integers $\cN$ congruent to $s$ modulo $K(k)$ can be written as in \eqref{representation}, thereby establishing a variant of Waring-Goldbach problem with \PS primes for $k\ge3$.  For $k=2$, it is shown in \cite{ZhaZai} that every sufficiently large integer $N \equiv 5 \pmod{24}$ can be written as in \eqref{representation} with $s=5$, provided that $1 < c < \tfrac{256}{249}$, while for $k=1$,  it follows from \cite{Kumchev} that every sufficiently large \emph{odd} integer can be written as in \eqref{representation} with $s=3$, provided that $1 < c < \tfrac{53}{50}$. 

Following the proof of the main theorem of \cite{ZhaZai}, the current range of $c$ in Theorem \ref{thm1} for the case $k=3$ can be improved. We shall leave this to a subsequent paper.

In analogy to Waring-Goldbach Problem, one can define $H_c (k)$ to be the least integer $s$ such that every sufficiently large integer congruent to $s$ modulo $K(k)$ can be expressed as in \eqref{representation}. Following the proof of Theorem 1 %with different major and minor arcs together with 
and using the methods in Hua's book \cite[\S 9]{Hua}, one may conclude that, for large $k$, $H_c (k)$ is bounded above by $4k\log k (1+o(1)) $, when $c$ lies in a slightly larger range than that of Theorem \ref{thm1}. However, coupling our results with the recent improvements of Wooley and Kumchev \cite{KumWoo1,KumWoo2} on Waring-Goldbach problem, we intend to futher improve this bound in an another paper.  	

The range of $c$ in Theorem \ref{thm1} is determined by three different estimates for exponential sums; van der Corput's estimate in Lemma \ref{corput} for $k=3$, Heath Brown's new estimate in Lemma \ref{HB} for $3 < k < 12$, and finally our estimate in Lemma \ref{ShiftLemma} for $k \ge 12$.

\begin{remark} %\label{thm1remark}
	Using Wooley's result \cite[Theorem 4.1]{Wooley} in the light of recent developments on Vinogradov's Mean Value Theorem by Bourgain, Demeter and Guth in \cite[Theorem 1.1]{BoDeGu}, it follows that the smallest exponent satisfying \eqref{2t_momentofT1} is
	\eqs{
		\renewcommand{\arraystretch}{1.2}
		\begin{tabular}{|*{11}{c|}}
			\hline
			$\mathbf{k}$ & 3 & 4 & 5 & 6 &7 &8 &9 &10&11 &12 \\ \hline
			$\mathbf{2t}$ &8&16&24&34&48&62&78&98&118&142 \\ \hline
	\end{tabular}}
	while for $k > 12$, it follows from \cite[Theorem 11]{Bourgain} that $2t$ can be chosen as the smallest even integer no smaller than
	\eqs{k^2 + 1 - \max_{s\le k} \Big \lceil s \frac{k-s-1}{k-s+1}\Big \rceil,}
	and for large $k$, $2t$ can be taken as large as $k^2 - k + O(\sqrt k)$.
\end{remark}

\section{Preliminaries and Notation} 
%\label{prelim}
\subsection{Notation }
%The notation $\| x\|$ is used to denote the distance from the real number $x$ to the nearest integer, that is, \comm{is this used?} $$ \| x\|=\min_{n\in\Z}|x-n|\qquad(x\in\R). $$
Throughout the paper, $k$, $m$ and $n$ are natural numbers with $k\ge 3$, and $p$ always denotes a prime number. We write $n \sim N$ to mean that $N < n \le 2N$. Furthermore, $c>1$ is a fixed real number and we put $\delta = 1/c$. 

Given a real number $x$, we write $e(x)=e^{2\pi ix}$, $\{x\}$ for the fractional part of $x$, $\fl x$ for the greatest integer not exceeding $x$. We write $\cL = \log \cN^{1/k}$.%, and $\ceil x$ for the least integer not smaller than $x$. 

For any function $f$, we put 
\eqs{\Delta f (x) = f \( -(x+1)^\delta \) - f ( - x^\delta), \qquad (x>0).}

We recall that for functions $F$ and real nonnegative $G$ the notations $F\ll G$ and $F=O(G)$ are equivalent to the statement that the
inequality $|F|\le \alpha G$ holds for some constant $\alpha>0$.  If $F\ge 0$ also, then $F\gg G$ is equivalent to $G\ll F$.  We also write $F\asymp G$ to indicate that $F\ll G$ and $G\ll F$. In what follows, any implied constants in the symbols $\ll$ and $O$ may depend on the parameters $c, \eps, k, s, t$, but are absolute otherwise. We shall frequently use $\eps$ %(not $\epsilon$) 
with a slight abuse of notation  to mean a small positive number, possibly a different one each time. 

%For any subset $\cS$ of integers and a real number $x$, $\cS (x)$ denotes the subset $\cS \cap [1,x]$ of $\cS$%, and $\#\cS(x)$ the number of elements of $\cS(x)$. Let $\cP$ denote the set of primes. Put
%\eqs{\Sc = \sum_{p \in \cP_c (X)} e(\alpha p^k), \qquad \Tc = \sum_{p \in \cP (X)}\delta p^{\delta-1} e(\alpha p^k).} 

Finally we put
\eqs{
	\Sc =\su{p \le X\\ p \in \cP_c} e(\alpha p^k), \qquad \Tc = \sum_{p \le X}\delta p^{\delta-1} e(\alpha p^k).}

\subsection{Preliminaries}
\subsubsection{Results related to PS sequences}
The characteristic function of the set $\sh = \{ \fl{m^c} : m \in \N \}$ is given by 
\eqn{\label{PScharacterization}
	\fl{- n^\delta } - \fl{ - (n+1)^\delta } = \left\{ \begin{array}{l@{,\quad}l}
		1 & \text{if } n \in \sh, \\[1mm]
		0 & \text{otherwise.}
	\end{array}
	\right.
}
Putting $\psi (x) = x - \fl x - 1/2$ we obtain 
\bsc{ \label{PSapproximation}\fl{- n^\delta } - \fl{ - (n+1)^\delta } &= (n+1)^\delta - n^\delta + \Delta \psi (n) \\
	&= \delta n^{\delta-1} + O (n^{\delta-2}) + \Delta \psi (n). }
The following result due to Vaaler gives an approximation to $\psi (x)$.
\begin{lemma}[{\cite[Appendix]{GraKol}}] \label{Vaaler}
	There exists a trigonometric polynomial 
	\eqs{\psi^\ast (x) = \sum_{1 \le \mid h \mid \le H} a_h e(hx ), \qquad \qquad (a_h \ll |h|^{-1})}
	such that for any real $x$,
	\eqs{| \psi(x) - \psi^\ast (x) | \le \sum_{|h|<H} b_h e(hx), \qquad\qquad (b_h \ll H^{-1}).}
\end{lemma}

\subsubsection{Definitions related to Waring-Goldbach Problem}
Put
\bsc{\label{defs}
	S(a,q) & = \su{1 \le x \le q \\ (x,q)=1} e \bigl( ax^k/q \bigr),  \\
	S_m (q) & = \su{1 \le a \le q \\ (a,q)=1} \bigl( \varphi(q)^{-1} S(a,q) \bigr)^s e(-ma/q),  \qquad (s \in \N, m \in \Z) \\
	\fS (m)  &= \sum_{q \ge 1} S_m (q), \\
	\cJ (z) & = \int_2^{N^{1/k}} \frac{\delta x^{\delta-1} e(zx^k)}{\log x} dx,  \quad \cI (z) = \int_0^{N^{1/k}} \delta x^{\delta-1} e(zx^k)  dx, \\
	v(z) & = \varphi (q)^{-1} S(a,q)  \cJ (z).	
}
By \cite[Lemma 8.5]{Hua} the estimate
\eqn{\label{S(aq)Bound} 
	S(a,q) \ll q^{1/2+\eps} }
holds for $\gcd(a,q)=1$. By the substitution $y =z x^k$ and the trivial estimate, it easily follows that
\eqn{\label{I(z)bound}
	\cI(z) \ll \min \( N^{\delta/k} , |z|^{-\delta/k}\).}

\begin{definition}[Major and Minor Arcs]
	For fixed $\kappa >0$, define 
	\eqs{\M (a,q) = \{ \alpha \in \R : | q\alpha - a| \le \cL^\kappa X^{-k} \}.}
	Let $\M$ be the union of all $\M(a,q)$ where $a, q$ are coprime integers such that $1 \le a \le q \le \cL^\kappa$. Note that the sets $\M(a,q)$ are pairwise disjoint and are contained in the unit interval $\cU_\kappa = ( \cL^\kappa X^{-k},1+\cL^\kappa X^{-k}]$. Put $\m = \cU_\kappa \setminus \M$. 
	
\end{definition}

\subsection{Standard Lemmas}

\begin{lemma}[{\cite[Theorem 1]{Heathbasgan}}] \label{HB} 
	Let $k \ge 3$ be an integer, and suppose that $f: [0,N] \to \R$ has continuous derivatives of order up to $k$ on $(0,N)$. Suppose further that  $0 < \lambda_k \le  f^{(k)} (x) \le A \lambda_k$ for $x \in (0,N)$. Then,
	\mult{\sum_{n \le N} e(f(n)) \\ \ll_{A,k,\eps} N^{1+\eps} \( {\lambda_k}^{1/k(k-1)}+ N^{-1/k(k-1)}+ N^{-2/k(k-1)} \lambda_k^{-2/k^2(k-1)} \).
	}
\end{lemma}
\begin{lemma}[{\cite[Theorem 2.8]{GraKol}}] \label{corput}
	Let $q$ be a positive integer. Suppose that $f$ is a real valued function with $q + 2$ continuous derivatives on some interval $I$. Suppose also that for some $\lambda  > 0$
	and for some $\alpha > 1$, 
	\eqs{\lambda \le  | f^{(q+2)} (x) | \le \alpha \lambda}
	on $I$. Let $Q = 2^q$. Then,
	\mult{\sum_{n \in I} e(f(n)) \ll \\
		|I| (\alpha^2 \lambda)^{1/(4Q-2)} + |I|^{1-1/(2Q)}  \alpha^{1/(2Q)} + |I|^{1 - 2/Q + 1/Q^2} \lambda^{-1/(2Q)}     
	}
	where the implied constant is absolute.
\end{lemma}

%\begin{lemma} [{\cite[Thm 8.20]{Iwaniec}}]\label{Corput}Let $f(x)$ be real and have continuous derivatives up to the $k$-th order, where $k \ge 4$. Let $\lambda_k \le f^{(k)} (x) \le h \lambda_k$ (or the same for $-f^{(k)} (x)$). Let $b-a\ge1$, $K=2^{k-1}$. Then,\eqs{\sum_{a < n \le b} e(f(n)) \ll h^{2/K} (b-a) \lambda_k^{1/(2K-2)} + (b-a)^{1-2/K} \lambda_k^{-1/(2K-2)},}where the implied constants are independent of $k$.\end{lemma}

\begin{lemma} \label{enlargeint}
	Assume $I_1$ is a subinterval of an interval $I$ with $|I_1| > 1$, and $g(x)$ is defined on $I$. Then,  
	\eqs{\sum_{n \in I_1} e\( g(n)\) \ll \log (1+|I|) \sup_{\gamma \in [0,1]} \bigg| \sum_{n \in I} e\(g(n) + \gamma n\)  \bigg|.} 
\end{lemma}
\begin{proof}
	The result follows upon taking the supremum over all $\gamma \in [0,1]$ in
	\eqs{\sum_{n \in I_1} e\( g(n) \) = \int_0^1 \sum_{n \in I} e\( g(n) +\gamma n \)  \sum_{m \in I_1} e(-\gamma m) d\gamma,}
	and using the fundamental inequality
	\eqs{
		\int_{0}^{1} \bigg|	\sum_{m \in I_1} e(-\gamma m) \bigg| d\gamma  \ll \int_{0}^{1} \min \bigg\{|I_1|, \frac{1}{|| \gamma ||} \bigg\} d\gamma \ll \log(1 + |I|)
	}
	where $||\gamma||=\min_{n \in \Z} |n-\gamma|$.
\end{proof}

\begin{lemma}[{\cite[Theorem 5]{Bourgain}}] \label{Weyl}
	Let $k\ge3$ be an integer, and let $\alpha_1, \ldots, \alpha_k \in \R$. Suppose that there exists a natural
	number $j$ with $2 \le j \le k$ such that, for some $a \in \Z$ and $q \in \N$ with $(a,q)=1$, one
	has $|\alpha_j -a/q | \le q^{-2}$% and $q \le X^j$
	. Then,
	% there exists a $\sigma (k)$ such that 
	\eqs{\sum_{1 \le n \le N} e\(\alpha_1 n + \cdots + \alpha_k n^k\)  \ll N^{1+ \eps} \( q^{-1} + N^{-1} + qN^{-j} \)^{1/k(k-1)}.}
\end{lemma}

The following result can be deduced from \cite[Prop. 13.4]{Iwaniec}.
\begin{lemma}[Vaughan's Identity]
	\label{VaugID}
	Let $u, v \ge 1$ be real numbers. If $n > v$ then,
	\eqs{\lm (n) = \su{ab = n \\ a \le u} \mu (a) \log b - \su{ab = n \\ a>v, b >u} \lm(a) \su{d \mid b \\d \le u} \mu(d) - \su{abc=n \\ b \le u,  a \le v} \mu(b) \lm (a) .}
\end{lemma}

\begin{lemma} \label{J2I}
	For any nonzero $\beta \in \R$,
	\eqs{ %\label{JminusI}
		\cJ (\beta) - \cL^{-1} \cI (\beta) \ll \frac{\cN^{\delta/k}}{\cL^{\kappa+2} } + \min\{ \cN^{\delta/k}, |\beta |^{-\delta/k}\}\frac{\log \log \cN}{\cL^2}. 
	}
	
\end{lemma}

\begin{proof}
	Using trivial estimate 
	\eqs{
		\int_2^J \delta x^{\delta-1} e(\beta x^k) \( (\log x)^{-1} - \cL^{-1} \)dx \ll \frac{J^{\delta/k}}{\log J},
	} %+ \frac{\cN^{\delta/k} \log\( \cN^{1/k} / B\)}{\cL \log B}
	for any $2 < J < \cN^{1/k}$. By partial integration
	\eqs{
		\int_J^{N^{1/k}} \delta x^{\delta-1} e(\beta x^k) \( (\log x)^{-1} - \cL^{-1} \) dx \ll \frac{\log \(\cN^{1/k}/J\)}{\cL \log J} \sup_{J < t \le N^{1/k}} |\Phi(\beta,t)| 
	}
	where
	\eqs{
		\Phi(\beta,t) =\int_2^t \delta x^{\delta-1} e(\beta x^k) dx \ll \min\{ t^\delta, |\beta |^{-\delta/k}\} 
	}
	uniformly for $2<t \le N^{1/k}$. Choosing $J = \cN^{1/k} (\log \cN)^{-c(\kappa+1)}$ and combining the above estimates completes the proof.
\end{proof}

\subsection{Exponential sum estimates}

This part constitutes the backbone of the entire paper and is to be used in the proof of Theorem 1.

\begin{lemma} \label{ShiftLemma}
	%Assume $k\ge 3$, $D>0$, and $g(x) \in \R[x]$ is a polynomial of degree not exceeding $k$. Let $I$ be a subinterval of $(N,2N]$. Then, the estimate 
	Assume $k\ge 3$, $D>0$, and $g(x) \in \R[x]$ is a polynomial of degree not exceeding $k$. Let $\ell \ge k+1$ be an integer. Then, the estimate 
	\begin{multline*} 
		\sum_{n \in I} e\(g(n) + D n^\delta\) \\
		\ll N ^{1+\eps} \( (DN^{\delta-k-1})^{\sigma} + (DN^\delta)^{\frac{\sigma}{\ell+1}}N^{-\sigma}  + (DN^\delta)^{-\frac{\ell-k}{\ell+1}\sigma} \)
	\end{multline*}
	holds with $\sigma^{-1} = \ell(\ell-1)$ or with $\sigma^{-1} = 2^k$ whenever $\ell=k+1$, for any subinterval $I$ of $(N,2N]$, where the implied constant depends only on $\eps,k$ and $\ell$.
\end{lemma}
\begin{proof}
	We shall first bound the sum 
	\eqs{\sum_{n \sim N} e( g(n) + D n^\delta)  }	
	for an arbitrary $g(x) \in \R[x]$ with $\deg g \le k$, and the result will follow by Lemma \ref{enlargeint}. 
	
	We can assume that $2^{\ell+1} < D N^\delta <  N^{k+1}$, since otherwise the claimed estimate holds trivially. Put $f(x) = g(x) + D x^\delta$. For $m \in \Z$ with $1 \le m \le M < N/2$, 
	\eqs{\sum_{n \sim N} e(f(n)) = \sum_{n \sim N} e\( f(n+m) \)  + O(m).}
	Thus, summing over $m \in [1, M]$, 
	\eqs{\sum_{n \sim N} e(f(n)) \ll \frac 1 M \sum_{n \sim N} \bigg| \sum_{1 \le m \le M} e\( f(n+m) \)\bigg| + M. }
	Let $R_j (x) = (1+x)^\delta - F_j (x)$, where $F_j (x)= \sum_{0\le i \le j} \tbinom \delta i x^i$ is the $j$th Taylor polynomial of $(1+x)^\delta$. Then, taking  $x = m/n$,
	\bs{f(n+m) 
		&= g(n+m) + Dn^\delta \( F_{\ell} (m/n) + R_{\ell} (m/n) \) \\
		&= P_{\ell} (m) + D n^\delta R_{\ell} (m/n)} 
	where $P_{\ell} (x) = \sum_{i=0}^{\ell} c_i x^i \in \R[x]$, $c_{k+1} = CDn^{\delta-k-1}$, and $0<|C|<1$. Noting that $R_{\ell}' (x) \ll |x|^{\ell}$ uniformly for $|x| \le M/N < 1/2$, we derive by partial integration and Lemma \ref{enlargeint} that 
	\begin{multline*}
		\sum_{m \le M} e\( f(n+m) \) %&\ll \Bigl(1+ DN^\delta (M/N)^{\ell + 1} \Bigr) \max_{M'\le M} \bigg| \sum_{1\le m' \le M'} e\( P_{\ell} (m') \) \bigg| 
		\\ \ll \Bigl(1+ DN^\delta (M/N)^{\ell + 1} \Bigr) \sup_{\gamma \in [0,1)} \bigg| \sum_{1 \le m \le M} e\( P_{\ell} (m)  + \gamma m \) \bigg| \log M.
	\end{multline*} 
	Note that $|c_{k+1} \pm1/q | \le q^{-2}$, where $q=\fl{|c_{k+1}|^{-1}} \ge 1$ since  $|c_{k+1}| < 1$. Then, taking $\ell = k+1$ and applying Weyl's inequality (cf. \cite[Lemma  2.4]{Vaughan})  yields for \emph{any} $\gamma \in \R$ that
	\eqs{\sum_{1 \le m \le M} e\( P_{k+1} (m)  + \gamma m \) \ll M ^{1+\eps} \( q^{-1} + M^{-1} + qM^{-k-1} \)^{2^{-k}},}
	while it follows from Lemma \ref{Weyl} that for arbitrary $\ell \ge k+1$, 
	\eqs{\sum_{1 \le m \le M} e\( P_{\ell} (m)  + \gamma m \) \ll M ^{1+\eps} \( q^{-1} + M^{-1} + qM^{-k-1} \)^{1/\ell(\ell-1)}.}
	%Furthermore, for $n \sim N$, we have $q \asymp N^{k+1-\delta} D^{-1}$. 
	In either case, we choose $M = N (DN^\delta)^{-\frac{1}{\ell+1}}$ so that we have $1 < M < N/2$, and thus we obtain
	\eqs{\sum_{n \sim N} e(f(n)) \ll N ^{1+\eps} \( q^{-1} + M^{-1} + qM^{-k-1} \)^{\sigma} + M. }
	Using the definitions of $M$ and $q$, and the fact that $\sigma < (\ell - k)^{-1}$, we see that the contribution of $N^{1+\eps} (qM^{-k-1})^{\sigma}$ is already larger than $M$, thus $M$ can be eliminated, and the result follows. 
\end{proof}

\begin{lemma} \label{TypeII}
	%Let $g(t) \in \R[t]$ be a polynomial with $\deg g \le k$. 
	Uniformly for any complex numbers $a_n, b_m$ with $|a_n|, |b_m|\le 1$, and $g(t) \in \R[t]$ of degree not exceeding $k$, 
	\eqs{S(x,y) = \su{m \sim x  \\   n \sim y \\ mn \sim N } a_n b_m e(hn^\delta m^\delta+ g(mn)) \ll N^{1+\eps}\min\{ S_1, S_2, S_3\},
	}
	where
	\eqn{\label{type2}
		\begin{split}
			S_1 &=  (|h| N^{\delta-1} x^{-\ell})^{\frac{\sigma/2}{\sigma + \ell+1}}   + (|h|N^{\delta-1}x^{-k})^{\frac{\sigma/2}{1 + \sigma}} \\  
			& \qquad + (|h|N^\delta)^{-\frac{\ell-k}{\ell+1}\sigma/2} 
			+(x/N)^{1/2} + x^{-\frac{\ell-k}{\ell-k+1}\sigma/2},  \\
			S_2 &= (x/N)^{1/2}+    (|h|N^{\delta})^{-1/k(k+1)^2} + x^{-1/2k(k+1)} \\
			& \qquad +(|h|N^{\delta-1}x^{-k})^{1/2(k^2+k+1)},  \\
			%\eqs{S_3 = X^{\frac 1 {56}} (|h| N^\delta)^{-\frac 1 {28}} + X^{-\frac 1 {16}} + Y^{-1/2} + (|h| N^{\delta-1} X^{-3})^{\frac 1 {30}}.}		
			S_3  &=   x^{-2^{-k-1}}+ (x/N)^{1/2} + x^{2^{1-2k}-2^{1-k}} (|h| N^\delta x^{-1-k})^{-2^{-k-1}} \\
			&\qquad + (|h| N^{\delta-1} x^{-k})^{\frac 1 {2^{k+2} - 2} } . 
		\end{split}
	}
\end{lemma}
\begin{proof}
	We may assume that $|h|N^{\delta-1}x^{-k}<1$; otherwise, the assertion is trivial. Applying Weyl-van der Corput inequality (cf. \cite[Lemma 2.5]{GraKol}) we see that
	\eqs{S^2(x,y) \ll \frac{(xy)^2}{Q}+\frac{xy^2}{Q}\sum_{ 1 \le |q| \le Q} \max_{y < n, n+q \le 2y} | \Ga(q,n,x)| }
	where $1 \le Q \le y$ is to be chosen optimally, and 
	\eqn{\label{Gammaqnx}
		\Ga(q,n,x) = \su{m \in I} e \(h ((n+q)^\delta-n^\delta)m^\delta+ g\( (n+q)m \) - g( nm ) \),}
	where $I \subseteq (x,2x]$ is an interval determined by the conditions $m \sim x$, $nm \sim N$, and $(n+q)m \sim N$. 
	
	If we apply Lemma \ref{ShiftLemma} with $D = |h \((n+q)^\delta-n^\delta\)|$, noting that $D \asymp |hq|y^{\delta-1}$, we obtain 
	\begin{multline*}
		\Ga(q,n,x) \ll x^{1+\eps} \bigl( (|hq|N^{\delta-1}x^{-k})^{\sigma} \\
		+ (|hq|N^{\delta-1}x^{-\ell})^{\frac{\sigma}{\ell+1}}  + (|hq|N^{\delta-1}x)^{-\frac{\ell-k}{\ell+1}\sigma} \bigr).
	\end{multline*}
	Inserting this estimate above and summing over $q$ yields 
	\begin{multline*}
		S^2(x,y) (xy)^{-2-\eps} \ll Q^{-1} + (Q|h|N^{\delta-1}x^{-k})^{\sigma} \\ +  (Q|h|N^{\delta-1}x^{-\ell})^{\frac{\sigma}{\ell+1}} + (Q|h|N^{\delta-1}x)^{-\frac{\ell-k}{\ell+1}\sigma}.
	\end{multline*}
	Using \cite[Lemma 2.4]{GraKol} to choose $1 \le Q \le y$ optimally, we conclude that
\eqs{\begin{split}
		S^2(x,y) (xy)^{-2-\eps} \ll  (|h|N^{\delta-1}x^{-\ell})^{\frac{\sigma}{\sigma+ \ell+1}} + (|h|N^{\delta-1}x^{-\ell})^{\frac{\sigma}{\ell+1}}  
		+ xN^{-1} \\
		\hfill + (|h|N^{\delta-1}x^{-k})^{\sigma} + (|h|N^{\delta-1}x^{-k})^{\frac{\sigma}{1 + \sigma}} \\ +  (|h|N^\delta)^{-\frac{\ell-k}{\ell+1}\sigma}  %\\
		+ x^{-\frac{\ell-k}{\ell-k+1}\sigma}  + (x^{-k-1})^{\frac{\ell-k}{2\ell + 1- k}\sigma}. 
\end{split}}
	Since $|h|N^{\delta-1}x^{-k}<1$, we can eliminate the second and the fourth terms, and the last term is smaller than the penultimate one. 
	
	If we instead apply Lemma \ref{HB} (with $k+1$ in place of $k$) to \eqref{Gammaqnx}, we obtain 
	\begin{multline*}
		\Ga(q,n,x) \ll x^{1+\eps} \Bigl( (|hq|N^{\delta-1}x^{-k})^{1/k(k+1)} \\
		+ x^{-1/k(k+1)}  + (|hq|N^{\delta-1}x)^{-2/k(k+1)^2} \Bigr),
	\end{multline*}
	which yields 
	\begin{multline*}
		S^2(x,y) (xy)^{-2-\eps} \ll Q^{-1} + (Q|h|N^{\delta-1}x^{-k})^{1/k(k+1)} \\ + x^{-1/k(k+1)}  + (Q|h|N^{\delta-1}x)^{-2/k(k+1)^2}.
	\end{multline*}
	Using \cite[Lemma 2.4]{GraKol} once again, we conclude that
	\begin{multline*}
		S^2(x,y) (xy)^{-2-\eps} \ll x/N+  (|h|N^{\delta-1}x^{-k})^{1/k(k+1)}+ x^{-1/k(k+1)}  \\ +  (|h|N^{\delta})^{-2/k(k+1)^2} + x^{-2/k(k+3)}+(|h|N^{\delta-1}x^{-k})^{1/(k^2+k+1)}  . 
	\end{multline*}		
	Since $|h|N^{\delta-1}x^{-k}<1$, we eliminate the second term. 
	
	Finally, if we apply van der Corput's result, Lemma \ref{corput}, to \eqref{Gammaqnx} and carry on similar calculations as above, we derive the desired estimate. 
\end{proof}

\begin{lemma} \label{ST}
	For any $\eps>0$, and $c \in (1,2)$, 
	\eqs{
		\begin{split}
			S_{c,3}(\alpha,X) &= T_{c,3}(\alpha,X) +  
			O\Bigl(  X^\eps\max\Big\{ X^{\frac{76\delta+77}{156}}, X^{\frac{79\delta+75}{157}} \Big\}\Bigr) \\
			\Sc &= \Tc + O\( X^{(1+\delta)\frac{\nu-1}{2\nu-1}+\eps}\), \quad k \ge 4
		\end{split}
	}
	holds uniformly for $\alpha \in \R$, where $\nu$ is given by \eqref{nu}.
\end{lemma}
\begin{proof}
	By \eqref{PScharacterization}, \eqref{PSapproximation} and Merten's Theorem (see \cite[Equation (2.15)]{Iwaniec})
	\eqn{\label{STalpha}
		\Sc = \Tc + \sum_{p \le X} e(\alpha p^k) \Delta \psi (p) + O(\log \log X).}
	In order to bound the middle term on the right, we divide the range of summation $[2,X]$ into dyadic intervals of the form $(N,2N]$. Applying Lemma \ref{Vaaler} on each such interval, we see that
	\eqs{\su{p \sim N} e(\alpha p^k) \Delta (\psi - \psi^\ast) (p) \ll H_N^{-1} \sum_{|h|<H_N} \bigg| \sum_{n \sim N} e(h n^\delta) \bigg|. }
	%it is easy to show that (cf. \cite[\S4.6]{GraKol}) \eqs{\sum_{p \sim N} e(\alpha p^k) \Delta \psi (p) \ll A (N) + B (N)}
	%and \eqn{\label{BN} B ( N) = N^{\delta-1} \sum_{1 \le |h| \le H_N} \max_{N < N' \le 2N} \bigg| \sum_{N < n \le N'} e(\alpha n^k + h n^\delta) \bigg|. }
	Using the exponent pair $(1/2,1/2)$ (cf. \cite[Ch.~3]{GraKol}) we obtain the estimate
	\eqs{\sum_{n \sim N} e(h n^\delta) \ll |h|^{1/2}N^{\delta/2} + |h|^{-1} N^{1-\delta}  \qquad \qquad (h \neq 0) }
	so that
	\eqn{\label{PsitoPsiStar} 
		\su{p \sim N} e(\alpha p^k) \Delta (\psi - \psi^\ast) (p) \ll  NH_N^{-1}  + B^{1/2} + H_N^{-1} \log H_N N^{1 - \delta},}
	where $B = H_N N^\delta$.
	%From now on, we shall write  \eqs{ %\label{H}  H_N = N^{1-\delta + \nu},} and choose $\nu$ optimally. Thus, we have obtained so far that  \eqn{\label{ANbound} A(N) \ll  NH_N^{-1}  + H_N^{1/2} N^{\delta/2} + H_N^{-1} \log H_N N^{1 - \delta} \ll N^{\delta - \nu} + N^{(1+ \nu)/2}. }
	
	Next, we turn to the sum involving $\psi^\ast$. First using partial summation and then introducing von Mangoldt function we obtain 
	\eqs{\su{p \sim N} e(\alpha p^k) \Delta \psi^\ast (p) \ll \frac 1 {\log N} \max_{N' \le 2N} \bigg| \sum_{N < n \le N'} \Delta \psi^\ast (n) e(\alpha n^k)  \lm (n) \bigg| + \sqrt N. 
	} 
	Recalling the definition of  $\psi^\ast$ it is not too hard (see \cite[4.6]{GraKol}) to derive that
	\bsc{\label{BN}
		&\su{p \sim N} e(\alpha p^k) \Delta \psi^\ast (p) \ll \Theta(N)+ \sqrt N, \\  
		& \Theta(N) = \frac{N^{\delta-1}}{\log N} \sum_{1 \le |h| \le H_N} \max_{N' \le 2N} \bigg| \sum_{N < n \le N'} e(\alpha n^k + h n^\delta)  \lm (n) \bigg| .}
	Assume that $u, v \ge 1$ are real numbers with $uv \le N$. Using Lemma \ref{VaugID} we write the inner sum on the right as $E_1 - E_2 - E_3$ where 
	\mult{E_1 = \su{n \le u} \mu (n) \su{N/n < m \le N'/n} e \( \alpha (nm)^k + h (nm)^\delta \) \log m  \\
		- \su{m \le u} \Bigl( \su{ab = m \\ b \le u,  a \le v} \mu(b) \lm (a) \Bigr) \su{N/m < n \le N'/m}  e \( \alpha (nm)^k + h (nm)^\delta \), }
	\eqs{E_2 = \su{N < nm \le N' \\ n>v, m>u} \lm(n) \Bigl( \su{d \mid m \\d \le u} \mu(d) \Bigr) e \( \alpha (nm)^k + h (nm)^\delta \), }
	and
	\eqs{E_3 = \su{N < nm \le N' \\ u < m \le uv} \Bigl( \su{ab = m \\ b \le u,  a \le v} \mu(b) \lm (a) \Bigr) e \( \alpha (nm)^k + h (nm)^\delta \) .}
	Note that 
	\eqn{\label{E1bdd}
		E_1 \ll \log N\sum_{1 \le n \le u} \max_{N<N'\le 2N} \bigg| \su{N/n < m \le N'/n} e \( \alpha (nm)^k + h (nm)^\delta \) \bigg|.}
	%Applying Lemma \ref{ShiftLemma} with $D = |h| n^\delta$ we see that the inner sum above is
	%\mult{\ll (N/n)^{1+\eps} \Bigl( N^{-(k+1-\delta)\sigma} |h|^{\sigma} n^{(k+1)\sigma} + N^{\(\frac{\delta}{\ell+1}-1\)\sigma} |h|^{\frac{\sigma}{\ell+1}} n^{\sigma} \\ + (|h| N^\delta)^{-\frac{\ell-k}{\ell+1} \sigma} \Bigr). }
	Thus, applying Lemma \ref{ShiftLemma} with $D = |h| n^\delta$ to the inner sum above and summing over $n \le u$ we obtain
	\mult{E_1 \ll N^{1+\eps} \Bigl( (|h|N^\delta)^{\sigma} (u/N)^{(k+1)\sigma-\eps} \\
		+ (|h|N^\delta)^{\frac{\sigma}{\ell+1}} (u/N)^{\sigma-\eps} + (|h| N^\delta)^{-\frac{\ell-k}{\ell+1} \sigma} \Bigr) .}
	Hence, the contribution to \eqref{BN} from $E_1$ is
	\eqn{ \label{E1contr}
		\ll B^{1+\eps} \Bigl( B^{\sigma} (u/N)^{(k+1)\sigma-\eps} + B^{\frac{\sigma}{\ell+1}} (u/N)^{\sigma-\eps} + B^{-\frac{\ell-k}{\ell+1} \sigma} \Bigr) .}
	
	Next, we estimate the bilinear sums $E_2$ and $E_3$. We first note that $E_2 \ll N^\eps \sum_{x, y} |S(x,y)|$, where 
	\eqs{S(x,y) = \sum_{m \sim x} b_m \su{n \sim y\\N < nm \le N' } a_n e \( \alpha (nm)^k + h (nm)^\delta \), }
	with $y>v, x>u$, $xy \asymp N$ and $|a_n|, |b_m| \le 1$. 
	Also, $E_3 \ll \log N \sum_{x,y} |S(x,y)|$
	with a similar bilinear sum $S(x,y)$, where $u<x \le uv$, $xy \asymp N$ and $|a_n|, |b_m| \le 1$. Applying the bound $S_1$ in Lemma \ref{TypeII} we obtain
	\begin{multline*}
		E_2 + E_3 \ll N^{1+2\eps} \Bigl( 
		v^{-1/2} + (uv/N)^{1/2} + u^{-\frac{\ell-k}{\ell-k+1}\sigma/2} + (|h|N^\delta)^{-\frac{\ell-k}{\ell+1}\sigma/2} \\ 
		+ (|h|N^{\delta-1}u^{-\ell})^{\frac{\sigma/2}{\sigma + \ell+1}} + (|h|N^{\delta-1}u^{-k})^{\frac{\sigma/2}{1 + \sigma}}  \Bigr) .
	\end{multline*}
	Choosing $v=(N/u)^{1/2}$, we see that the contribution of $E_2+ E_3$ to \eqref{BN} is
	\begin{multline}\label{E23contr}
		\ll B^{1 + 2\eps} \Bigl( 
		(u/N)^{1/4} + u^{-\frac{\ell-k}{\ell-k+1}\sigma/2} + B^{-\frac{\ell-k}{\ell+1}\sigma/2} \\ 
		+ (B N^{-1}u^{-\ell})^{\frac{\sigma/2}{\sigma + \ell+1}} 
		+ (B N^{-1}u^{-k})^{\frac{\sigma/2}{1 + \sigma}}  \Bigr) .
	\end{multline} 
	Combining \eqref{PsitoPsiStar}, \eqref{E1contr} and \eqref{E23contr} we conclude that
	\begin{multline*}
		\sum_{p \sim N} e(\alpha p^k) \Delta \psi (p) \ll 
		NH_N^{-1}  + B^{1-\frac{\ell-k}{\ell+1}\sigma/2+\eps} + B^{1/2} \\
		\quad+ B^{1+\eps} \Bigl( B^{\sigma} (u/N)^{(k+1)\sigma-\eps} + (u/N)^{1/4} + B^{\frac{\sigma}{\ell+1}} (u/N)^{\sigma-\eps} \\
		+ u^{-\frac{\ell-k}{\ell-k+1}\sigma/2} + (B N^{-1}u^{-\ell})^{\frac{\sigma/2}{\sigma + \ell+1}} 
		+ (B N^{-1}u^{-k})^{\frac{\sigma/2}{1 + \sigma}} \Bigr).
	\end{multline*}
	Note that the second term dominates the third. Since the first two terms are independent of $u$, we choose%provided \eqref{bound4ufromE1} holds and $u > N^{1/k}$, both of which will be satisfied with our choice in what follows. We first put 
	\eqs{H_N = N^{1 - (1+\delta) \frac{1-A(\ell)}{2-A(\ell)}}, \qquad \quad A(\ell) = \frac{(\ell-k)\sigma}{2(\ell+1)} }
	so as to balance them first. Note that with this choice, we have $1 < H_N < N$, and
	\eqs{N H_N^{-1} = N^{(1+\delta) \frac{1-A(\ell)}{2-A(\ell)}}, \qquad B = H_N N^\delta = N^{\frac{1+\delta}{2-A(\ell)}}.}
	In order to minimize $NH_N^{-1}$,  we set 
	\eqs{A = A_k = \max_{\ell \ge k+1} A(\ell).}
	It follows by an easy computation that $\ell = \fl{3k/2}$ maximizes $A(\ell)$, and with this choice of $A$, we find by setting $u=N^{1/2}$ that all the remaining terms are smaller than $NH_N^{-1}$, and thus we conclude that
	\bsc{\label{klarge}
		\su{p \sim N} e(\alpha p^k) \Delta \psi (p) \ll N^{(1+\delta)\frac{1-A}{2-A} +\eps}.}
	
	Next, we estimate the inner sum in \eqref{E1bdd} using Lemma \ref{HB}. This gives
	\eqs{E_1 \ll N^{1+\eps} \Bigl( (hN^\delta)^{1/k(k+1)} (u/N)^{1/k} + (u/N)^{1/k(k+1)} + (hN^\delta)^{-2/k(k+1)^2} \Bigr),
	}
	whose contribution to \eqref{BN} is
	\eqn{\label{E1HB}
		\ll B^{1+\eps} \Bigl( B^{1/k(k+1)} (u/N)^{1/k} + (u/N)^{1/k(k+1)} + B^{-2/k(k+1)^2} \Bigr).
	}
	Using the bound $S_2$ in \eqref{type2}, we see that the contribution of $E_2+ E_3$ to \eqref{BN} is
	\multn{\label{E23HB}
		\ll B^{1 + 2\eps} \Bigl( (u/N)^{1/4} + B^{-1/k(k+1)^2} + u^{-1/2k(k+1)} \\
		+(BN^{-1}u^{-k})^{1/2(k^2+k+1)} \Bigr).}
	Combining \eqref{PsitoPsiStar}, \eqref{E1HB} and \eqref{E23HB} shows that \eqref{BN} is bounded by
	\mult{NH_N^{-1} +  B^{1-1/k(k+1)^2} + B^{1 +\eps} \Bigl( B^{1/k(k+1)} (u/N)^{1/k} + (u/N)^{1/k(k+1)}  \\ 
		+ u^{-1/2k(k+1)} + (BN^{-1}u^{-k})^{1/2(k^2+k+1)} \Bigr). 
	}
	Choosing $H_N$ to balance the first two terms again gives
	%\eqs{H_N = N^{\frac{1- \delta A}{1+ A}}} so that
	\eqs{NH_N^{-1} = N^{(1+\delta) \frac{1-C}{2-C}}, \qquad C = \frac 1 {k(k+1)^2}.
	}
	As before, for $u=N^{1/2}$, all the remaining terms are dominated by $NH_N^{-1}$. Thus,
	\eqn{\label{ksmall}
		\su{p \sim N} e(\alpha p^k) \Delta \psi (p) \ll N^{(1+\delta)\frac{1-C}{2-C} +\eps}.}
	
	One can easily check that for $3 \le k \le 11$, using Heath-Brown's result (Lemma \ref{HB}) gives a better estimate since $C > A$. For $k \ge 12$, however, $A > C$, which explains our choice in \eqref{nu}. 
	
	Finally, (only) for $k=3$, one can do slightly better than Heath-Brown's estimate by using van der Corput's estimate; namely, by Lemma \ref{corput} with $q=2$,
	it follows that
	\eqs{E_1 \ll  N \log N\( \( hN^\delta(u/N)^4\)^{\frac{1}{14}}+(u/N)^{\frac{1}{8}}+(hN^\delta)^{-\frac{1}{8}}N^{\frac{1}{16}}
		\)
	}
	whose contribution to \eqref{BN} is
	\eqn{\label{3E1} \ll B \log N \Bigl( 
		(B (u/N)^4 )^{\frac{1}{14}} + (u/N)^{\frac{1}{8}} + B^{-\frac 1 {8}} N^{\frac 1 {16}}
		\Bigr).}
	On the other hand, using $S_3$ in \eqref{TypeII}, we obtain for $k=3$, 
	\mult{E_2+E_3 \ll N^{1+\eps} \bigl( 
		u^{-\frac 1 {16}} + B^{\frac 1 {30}} u^{-\frac 1 {10}} + v^{-\frac{1}{2}} + B^{-\frac 1 {16}} (N/v)^{\frac{1}{32}} \\ +(uv/N)^{\frac{1}{2}}+B^{-\frac 1 {16}} (uv)^{\frac{1}{32}} 
		\bigr).}
	Choosing $v = \sqrt{N/u}$ and summing over $h$, the contribution from $E_2+ E_3$ is
	\eqn{\label{3E23}
		\ll  B^{1+\eps} \bigl( 
		u^{-\frac 1 {16}} + B^{\frac 1 {30}} N^{-\frac 1 {30}} u^{-\frac 1 {10}} + (u/N)^{\frac 1 4} + B^{-\frac 1 {16}} (uN)^{\frac 1 {64}}
		\bigr).}
	Combining \eqref{PsitoPsiStar}, \eqref{3E1} and \eqref{3E23}, the total contribution is 
	\mult{\ll NH_N^{-1} + B^{1-\frac 1 {8}} N^{\frac 1 {16}} + B^{1+\eps} \bigl( 
		u^{-\frac 1 {16}} + B^{\frac 1 {30}} N^{-\frac 1 {30}} u^{-\frac 1 {10}} 
		+ B^{-\frac 1 {16}} (uN)^{\frac 1 {64}} \\ 
		+  B^{\frac{1}{14}} (u/N)^{\frac{4}{14}}  + (u/N)^{\frac{1}{8}} \bigr).}
	Choosing $u$ optimally above by using \cite[Lemma 2.4]{GraKol} with $1 \le u \le N$, we have that \eqref{BN} is bounded by
	\mult{\ll NH_N^{-1} + B^{1+\eps} \Bigl( N^{-\frac 1 {16}} + B^{\frac 1 {30}} N^{-\frac 2 {15}} 
		+ B^{-\frac 1 {16}} N^{\frac 1 {64}} +  B^{\frac{1}{14}} N^{-\frac{4}{14}}  + B^{-\frac{1}{20}} N^{\frac{1}{80}}\\
		+ B^{\frac{1}{78}} N^{-\frac{4}{78}} +  N^{-\frac{1}{24}} + B^{-\frac{11}{222}} N^{\frac{1}{111}} + B^{\frac{7}{162}} N^{-\frac{8}{81}} + B^{\frac{1}{54}} N^{-\frac{2}{27}} \Bigr),}
	%\mult{B^{-\frac 1 {14}} + N^{-\frac 1 {16}} + B^{\frac 1 {30}} N^{-\frac 2{15}}  + B^{\frac 1 {14}} N^{-\frac 2 7} + N^{\frac 1 {112}} B^{-\frac 1 {28}} \\ + \( \( B^{-\frac 1 {14}} N^{\frac 1 {28}} \)^{\frac 2 7} \( B^{\frac 1 {14}} N^{-\frac 2 7} \)^{\frac 1 {28}} \)^{\frac{1}{\frac{2}{7} + \frac{1}{28}}}  + \( \( B^{-\frac 1 {14}} N^{\frac 1 {28}} \)^{\frac 1 {112}} \( N^{\frac 1 {112}} B^{-\frac 1 {28}} \)^{\frac 1 {28}} \)^{\frac{1}{\frac{1}{112} + \frac{1}{28}}}  \\ + \( B^{\frac 1 {14}} N^{-2/7} \)^{\frac 1 {16} \frac 1 {\frac 1{16}+ \frac{2}{7}}} + \( N^{\frac 1 {112}} B^{-\frac 1 {28}} \)^{\frac 1 {16} \frac 1 {\frac 1{16}+ \frac{1}{112}}} \\ +	\( \(  B^{\frac 1{30}} N^{-\frac 1 {30}} \)^{\frac{2}{7}} \( B^{\frac 1 {14}}  N^{-\frac 2 7} \)^{\frac 1 {10}} \)^{\frac{1}{\frac{1}{10} + \frac{2}{7}}}  +	\( \( B^{\frac 1{30}} N^{-\frac 1 {30}} \)^{\frac 1 {112}} \( N^{\frac 1 {112}} B^{-\frac 1 {28}} \)^{\frac 1 {10}} \)^{\frac{1}{\frac{1}{10} + \frac 1 {112}}}  	} 
	Here, only the first term has $H_N$ with a negative exponent. Balancing terms with an appropriate $1 \le H_N \le N$ using \cite[Lemma 2.4]{GraKol} again, \eqref{BN} is bounded by
	\mult{
		\ll N^{\eps}\Bigr(1 + N^{\frac {14\delta+1}  {16}} + N^{\frac {31\delta-4} {30}}  + N^{\frac {60\delta +1}{64}} +  N^{\frac{15\delta - 4}{14}}  + N^{\frac{76\delta +1}{80}}
		+ N^{\frac{79\delta-4}{78}}  + N^{\frac{24\delta-1}{24}} \\
		+ N^{\frac{211\delta+2}{222}} + N^{\frac{169\delta-16}{162}} + N^{\frac{55\delta-4}{54}} + N^{\frac{7\delta}{15}+\frac 1 2} + N^{\frac{31\delta+27}{61}} + N^{\frac{60\delta+61}{124}} + N^{\frac{15\delta+11}{29}} \\
		+ N^{\frac{76\delta+77}{156}}+ N^{\frac{79\delta+75}{157}} + N^{\frac{\delta}{2}+\frac{23}{48}}+ N^{\frac{211\delta+213}{433}}+ N^{\frac{169\delta+153}{331}}+ N^{\frac{55\delta+51}{109}}\Bigl)
		.}
	Comparing all the terms under the assumption that $\delta \in (1/2,1)$, we end up with \eqn{\label{k=3}
		\su{p \sim N} e(\alpha p^k) \Delta \psi (p) \ll N^{\eps} \max \big\{ N^{\frac{76\delta+77}{156}}, N^{\frac{79\delta+75}{157}} \big\}.}
	The result follows by inserting \eqref{klarge}, \eqref{ksmall} and \eqref{k=3} back to \eqref{STalpha}.
\end{proof}

\begin{lemma} \label{Sv}
	If $1 < c < 12/11$, then for any $\alpha \in \M(a,q)$ with $\gcd(a,q)=1$, $1 \le a \le q \le \cL^\kappa$ and sufficiently large $X$,  we have
	\eqs{\Sc  - v(\alpha-a/q) \ll X^\delta \exp(-C \sqrt {\log X}),}
	where $C>0$ is an absolute constant and the implied constant depends only on $\kappa$ and $k$. 
\end{lemma}
\begin{proof}
	Combining \eqref{STalpha} with equations \eqref{PsitoPsiStar} and \eqref{BN}, in which we take $H_N = N^{1 - \delta + \eps}$, we see that
	\eqs{
		\Sc - \Tc \ll   \sum_{N=2^l \le X} \( N^{\delta-\eps} + \Theta (N)\).}
	The inner sum in the definition of $\Theta(N)$ can be written as 
	\eqs{\su{1 \le m \le q \\ (m,q)=1} e\(\frac{am^k}{q}\)\su{N < n \le N' \\ n \equiv m \mod q} e(\beta n^k + h n^\delta) \lm (n) + O(q\log N).}
	Removing $\e(\beta n^k)$ by partial summation this double sum is bounded by 
	\eqs{\su{1 \le m \le q\\ (m,q)=1} \( 1 + N^k \cL^\kappa X^{-k} q^{-1} \) \max_{N < N' \le 2N} \bigg| \su{N < n \le N' \\ n \equiv m \mod q} e(h n^\delta) \lm (n) \bigg|. }
	Applying the estimate given as the first equation on page 323 of \cite{cheboshap}, which is uniform both in $m$ and $q$, we derive that
	\bs{
		\Theta(N) 
		&\ll N^{\delta-1}\cL^{\kappa-1} \su{1 \le m \le q\\ (m,q)=1}  q^{-1} 
		\sum_{1 \le |h| \le H_N} \max_{N' \le 2N} \bigg| \su{N < n \le N' \\ n \equiv m \mod q} e(h n^\delta)  \lm (n) \bigg| \\
		& \ll N^{\delta} \exp(-c_1 \sqrt{\log N}) 
	}
	for an absolute constant $c_1>0$, and any fixed $1 < c < 12/11$. 
	
	Next, we deal with $\Tc$. Writing $\beta = \alpha -a/q$
	\eqs{\Tc  = \su{1 \le b \le q \\ (b,q)=1} e(ab^k/q) \su{p \le X \\ p \equiv b \mod q} \delta p^{\delta-1} e(\beta p^k) + O(\omega (q))}
	where $\omega (n)$ is the number of distinct prime divisors of $n$. It follows from Siegel-Walfisz theorem that
	\eqs{\su{p \le x \\ p \equiv b \mod q}  1= \frac{1}{\varphi(q)} \int_{2}^{x}\frac{dt}{\log t} + E(x), }
	uniformly for $q \le \cL^\kappa$, where $E(x) \ll x \exp(-c_2 \sqrt{\log x})$ for an absolute constant $c_2=c_2(\kappa)>0$ and large $x$. By partial integration we derive that
	\begin{multline*}
		\su{p \le X \\ p \equiv b \mod q} \delta p^{\delta-1} e(\beta p^k) = \int_{2^-}^X \frac{\delta x^{\delta-1} e(\beta x^k)}{\varphi(q) \log x} dx \\
		+ O \( X^{\delta-1} E(X)  + \cL^\kappa  \int_{2^-}^X |E(x)| x^{\delta-2} dx \).
	\end{multline*}							
	Using $E(x) \ll x$ when $x$ is small (say $x \le \sqrt X$) and the above bound for large $x$ in the last integral and inserting the result above we obtain  
	\eqs{\Tc = v(\alpha - a/q) + O_\kappa \( X^\delta \exp(-c_3 \sqrt {\log X}) \),}
	for sufficiently large $X$ and some positive absolute constant $c_3 < c_2$. Combining all the estimates above, the result follows. 
\end{proof}

\section{Proof of Theorem \ref{thm1}} 

Recall that, for a fixed $k \ge 3$ and $c>1$, $\Rc \cN$ is the number of representations of a positive integer $\cN$ as in \eqref{representation}. It can be rewritten as 
\eqs{ %\label{Rc(N)}
	\Rc \cN = \int_\cU \Sc^s e(-\alpha \cN) d\alpha,
}
where $\cU$ is any interval of unit length and $X = \fl{\cN^{1/k}}$.

\begin{lemma}[Major Arcs] \label{lem:major}
	Assume that $s \ge \max(5,k+1)$, and that  $1<c< \min\{ \frac{12}{11}, \frac{s}{k} \}$. Then, uniformly for integers $m$ with $1 \le m \le \cN$, and $\kappa > \tfrac{4s}{2s-9}$,
	\eqs{
		\int_{\M} \Sc^s  e(-\alpha m) d\alpha = \fS (m) m^{\delta s/k-1} \frac{\Gamma(1+\delta/k)^s}{\Gamma(s\delta/k)  \cL^s } + o\( \frac{\cN^{\delta s/k - 1}}{\cL^s} \).
	}
\end{lemma}
\begin{proof}
	By Lemma \ref{Sv} below, 
	\eqs{\Sc - v(\alpha-a/q) \ll X^\delta E(X),}
	where $E(X)= \exp(-c_2 \sqrt {\log X})$, uniformly for $\alpha \in \M(a,q)$ with $(a,q)=1$ and $1 \le a \le q \le \cL^\kappa$. Put $\beta = \alpha- a/q$. Then, 
	\eqs{\Sc^s - v(\beta)^s \ll (X^\delta E(X))^s + X^\delta E(X) |v(\beta)|^{s-1}. }
	Therefore, 
	\eqs{\sum_{q \le \cL^\kappa} \su{a \le q \\ (a,q)=1}  \int_{\M(a,q)} \bigl( \Sc^s - v(\alpha-a/q)^s  \bigr) e(-\alpha m) d\alpha = o \( X^{\delta s-k} \cL^{-s} \).}
	Furthermore,
	\eqs{\sum_{q \le \cL^\kappa} \su{a \le q \\ (a,q)=1} 
		\mathop{\int}_{\M(a,q)} v(\beta)^s  e(-\alpha m) d\alpha = \sum_{q \le \cL^\kappa} S_m (q) \mathop{\int}_{|q\beta| \le \frac{\cL^\kappa }{X^k}} \cJ (\beta)^s  e(-\beta m) d\beta.}
	Using Lemma \ref{J2I} together with \eqref{I(z)bound} and the bound $S_m (q) \ll q^{1-s/2+\eps}$ (which follows from \eqref{S(aq)Bound}), %$\cI (\beta) \ll \min\( X^{\delta/k, |\beta|^{-\delta/k} \)$
	we see that replacing the integral $\cJ(\beta)$ above by $\cL^{-1}\cI (\beta)$ introduces an error of size $o(X^{\delta s -k} /\cL^s)$.  
	%Let $R$ be the error term in approximating the integral $\cJ$ with $\cI$ in Lemma \ref{J2I}. Then, 
	%\begin{multline*} \mathop{\int}_{|q\beta| \le \frac{\cL^\kappa }{X^k}}  \cJ(\beta)^s  e(-\beta m) d\beta- \frac 1 {\cL^s} \mathop{\int}_{|q\beta| \le \frac{\cL^\kappa }{X^k}}  \cI(\beta)^s  e(-\beta m) d\beta \\ \ll \Biggl( \int_{|\beta| \le X^{-k}}  
	% + \mathop{\int}_{X^{-k} < |\beta| \le \frac{\cL^\kappa }{qX^k}}  \Biggr) \Bigl( \frac{|\cI(\beta)|^{s-1}}{\cL^{s-1}} R + R^s \Bigr)d\beta. \end{multline*}Summing over $q \le \cL^\kappa$ and using the bound $S_m (q) \ll q^{1-s/2+\eps}$, we obtain
	We can then extend the integral over $\beta$ to $\R$ with another permissible error. By \cite[Lemma 8]{Waringshap}, 
	\eqs{
		\int_\R \cI (\beta)^s  e(-\beta m) d\beta = m^{\delta s/k-1} \frac{\Gamma(1+\delta/k)^s}{\Gamma(s \delta/k)}.
	}
	Finally,  using
	\eqs{\sum_{q \le \cL^\kappa} S_m (q) = \fS (m) +  O \( \cL^{\kappa (2-s/2+\eps) } \)
	}
	completes the proof. 
\end{proof}
\begin{lemma}[Minor Arcs] \label{lem:minor}
	Assume that $\lambda>0$, and $t$ is an integer for which \eqref{2t_momentofT1} holds. Then,
	\eqs{\int_\m  |\Sc|^s d\alpha  \ll X^{\delta s - k} \cL^{2t-1- \lambda \delta (s-2t)+\eta} + X^{(s-2t)\theta + 2t -k+\eps},}
	provided that $\kappa \ge 2^{6k}(2+\lambda)$, where $\theta$ is the exponent of $X$ in the error term in Lemma \ref{ST}.
\end{lemma}
\begin{proof}
	Using equation \eqref{STalpha} we obtain
	\eqs{\int_\m |\Sc|^s d\alpha \ll I_1 + I_2 + O\( (\log \log X)^s \)}
	where 
	\begin{equation*}
		I_1 = \int_\m |\Tc|^s d\alpha, \qquad I_2 = \int_\m \bigg| \sum_{p \le X} e(\alpha p^k) \Delta \psi (p) \bigg|^s.
	\end{equation*}
	We first bound $I_1$. Let $ 2< J \le X$ be a constant to be determined. By partial integration 
	\eqs{\Tc \ll J^\delta + J^{\delta-1} \sup_{J< t \le X} \bigg| \su{J < p \le t} e(\alpha p^k) \bigg|.}
	Take $\alpha \in \m$. By Dirichlet's approximation theorem, one can find integers $a,~q$ with $1 \le a \le q \le \cL^{-\kappa} X^k$ such that $|\alpha - a/q | \le q^{-1} \cL^\kappa X^{-k}$. Since $\alpha \in \m$, we have $q > \cL^\kappa$. Writing $\alpha = a/q + \beta$, and using partial integration we obtain
	\eqs{\su{J < p \le t} e(\alpha p^k) \ll \sup_{J < y \le t} \bigg| \su{J < p \le y} e(ap^k/q)   \bigg| \( 1 + |\beta| t^k \).}
	Following the proof of Lemma \ref{enlargeint} and recalling that $y \le t \le X$,
	\eqs{\su{J < p \le y} e(ap^k/q) \ll \log X \sup_{\gamma \in [0,1)} \bigg| \sum_{p \le X} e(ap^k/q + \gamma p) \bigg|,}
	so that
	\eqs{\Tc\ll J^{\delta} + J^{\delta-1} \sup_{\gamma \in [0,1)} \bigg| \sum_{p \le X} e(ap^k/q + \gamma p) \bigg| \log X.%\( 1 + |\beta| X^k \)
	}
	By \cite[Theorem 10]{Hua} it follows for arbitrary $\lambda >0$ and any $\gamma \in \R$ that whenever $\kappa \ge 2^{6k}(2+\lambda)$,
	\eqs{\su{p \le X} e(a p^k/q +\gamma p) \ll X \cL^{-\lambda-1}.}
	Choosing $J = X \cL^{-\lambda}$ we conclude that
	\eqs{\Tc \ll X^\delta \cL^{-\lambda\delta}.}
	Using this bound together with H\"older's inequality yields
	\bs{I_1 &\le \sup_{\alpha \in \m} |\Tc|^{s-2t} \int_\m \bigg|\sum_{2^l = N \le X} \su{p\sim N} \delta p^{\delta-1} e(\alpha p^k) \bigg|^{2t} d\alpha \\
		& \ll (X^\delta \cL^{-\lambda\delta})^{(s-2t)} \cL^{2t-1} \sum_{N \le X} \int_0^1 \bigg| \su{p\sim N} \delta p^{\delta-1} e(\alpha p^k) \bigg|^{2t} d\alpha.  
	}
	By considering the underlying Diophantine equations we see that the last integral is 
	\eqs{\ll N^{2t(\delta-1)} \int_0^1 \bigg| \sum_{n \le N} e(\alpha n^k) \bigg|^{2t} d\alpha. % \ll N^{2t\delta -k} \cL^{\eta} .
	}
	Using \eqref{2t_momentofT1} we conclude that for some $\eta \ge0$,
	\eqn{\label{minorI1}
		I_1 \ll X^{\delta s - k} \cL^{2t-1- \lambda \delta (s-2t)+\eta}.}
	
	Next, we deal with $I_2$. Note that
	\eqs{I_2 \ll \sup_{\alpha \in \m} \bigg| \sum_{p \le X} e(\alpha p^k) \Delta \psi (p) \bigg|^{s-2t} \int_0^1 \bigg| \sum_{n \le X} e(\alpha n^k) \bigg|^{2t} d\alpha.}
	Using \eqref{STalpha} and then applying Lemma \ref{ST} together with \eqref{2t_momentofT1} we obtain
	\eqn{\label{minorI2}
		I_2 \ll X^{(s-2t) \theta + 2t -k+\eps}.}
	Combining \eqref{minorI1} and \eqref{minorI2}, the proof is completed. 
\end{proof}

%\hrulefill

%\eqs{\int\m \ll \( \( X^\delta \cL^{-\lambda\delta} \)^{s-2t} + X^{(1+\delta)\frac{\nu-1}{2\nu-1}(s-2t)} \) \int_0^1 |\Sc|^{2t} d\alpha}

%\hrulefill

The proof of Theorem \ref{thm1} can now be completed by taking $m = \cN$ in Lemma \ref{lem:major} and observing that taking $\lambda$ (and thus $\kappa$) sufficiently large in Lemma \ref{lem:minor} ensures that the contribution from minor arcs is $o(X^{\delta s - k} \cL^{-s})$ under the additional assumption in \eqref{rangeofc}. \medskip

\noindent{\bf Acknowledgments.} We would like to thank Professor T. Wooley for reading this manuscript. We would also like to thank the referee for carefully reading this paper and his/her useful comments on the organization of the paper that we believe made it more readable.

\end{document}